\documentclass[11pt]{amsart}
\usepackage{amsmath,amsthm,amssymb}
\usepackage{graphicx}
\usepackage{hyperref}
\hypersetup{
 setpagesize=false,
 bookmarksnumbered=true,%
 bookmarksopen=true,%
 colorlinks=true,%
 linkcolor=red,
 citecolor=blue,
}

\theoremstyle{plain}
\newtheorem{thm}{Theorem}[section]
\newtheorem{lem}[thm]{Lemma}
\newtheorem{prop}[thm]{Proposition}
\newtheorem{cor}[thm]{Corollary}

\theoremstyle{definition}
\newtheorem{defi}[thm]{Definition}

\newtheorem{rem}[thm]{Remark}
\newtheorem{qst}[thm]{Question}

\newcommand{\NN}{\mathbb{N}}

\newcommand{\QQ}{\mathbb{Q}}
\newcommand{\RR}{\mathbb{R}}
\newcommand{\CC}{\mathbb{C}}

\newcommand{\calT}{\mathcal{T}}
\newcommand{\calD}{\mathcal{D}}
\newcommand{\Int}{\mathrm{Int\,}}

\title[Minimal genus relative trisections of corks]{Minimal genus relative trisections of corks}
\author[Natsuya Takahashi]{Natsuya Takahashi}
\date{May 9, 2023.}
\address{Department of Pure and Applied Mathematics, Graduate School of Information Science and Technology, Osaka University, 1-5 Yamadaoka, Suita, Osaka 565-0871, Japan}
\email{nt-takahashi@ist.osaka-u.ac.jp}
\begin{document}
\begin{abstract}
In this paper, we prove that the trisection genus of the Akbulut cork is $3$ and construct infinitely many corks with trisection genus $3$. These results give the first examples of contractible $4$-manifolds whose trisection genera are determined except for the $4$-ball. We also give a lower bound for the trisection genus of a $4$-manifold with boundary. In addition, we construct low genus relative trisection diagrams of an exotic pair of simply-connected $4$-manifolds with $b_2 = 1$.
\end{abstract}
\maketitle

\section{Introduction}\label{sec:intro}

 A trisection is a decomposition of a $4$-manifold into three $4$-dimensional $1$-handlebodies. It was introduced by Gay and Kirby \cite{GayKir16} as a $4$-dimensional analogue of Heegaard splittings of $3$-manifolds.
 For compact $4$-manifolds with boundary, the notion of relative trisections was also introduced in \cite{GayKir16}.
 Later, it was studied in \cite{Cas16}, \cite{CasGayPin18_1}, \cite{CasGayPin18_2}, \cite{CasOzb19}, \cite{KimMil20}, and \cite{CasIslMilTom21} for example.
 A relative trisection diagram is a description of a relative trisection by three families of curves on a compact surface.
 In \cite{CasGayPin18_1}, Castro, Gay, and Pinz\'{o}n-Caiced established a natural correspondence between relative trisections and relative trisection diagrams. 
 By the correspondence, one can represent smooth structures of $4$-manifolds with boundary by relative trisection diagrams.

 A trisection genus is a fundamental invariant of smooth $4$-manifolds defined by trisections.
 For a $4$-manifold $X$, the trisection genus of $X$ is the minimal integer $g$ such that $X$ admits a (relative) trisection with the triple intersection surface of genus $g$.
 Meier and Zupan \cite{MeiZup17_1} classified closed, oriented, smooth $4$-manifolds with trisection genus at most $2$.
 However, in general, it is difficult to determine the trisection genus of a $4$-manifold, and the following question naturally arises.

\begin{qst}\label{qst}
 For a given smooth $4$-manifold, what is its trisection genus?
\end{qst}

 In this paper, we answer the above question for the Akbulut cork.

\begin{thm}\label{mainthm:Ac}
 The trisection genus of the Akbulut cork is $3$.
\end{thm}

 As far as we know, this is the first example of a contractible $4$-manifold whose trisection genus is determined except for the $4$-ball.
 A cork is a pair of a contractible $4$-manifold and a smooth involution on the boundary.
 By using corks, one can construct exotic (i.e., homeomorphic but not diffeomorphic) smooth structures of $4$-manifolds. 
 It is well known that the Akbulut cork is the first example of a cork (\cite{Akb91_1}).
 Details about corks will be described in Subsection~\ref{subsec:Corks}.
 In addition, we show the following theorem.

\begin{thm}\label{mainthm:corks}
There are infinitely many corks with trisection genus $3$.
\end{thm}

 This theorem is proved by constructing an infinite family $\{M_n\}_{n\in\NN}$ of trisected corks.
 (In fact, $M_1$ is diffeomorphic to the Akbulut cork.)
 The genus $3$ relative trisection diagram of $M_n$ is shown in Figure~\ref{fig:D_n-rtd}.
 We also give nice properties of $\{M_n\}_{n\in\NN}$ (see Theorems~\ref{thm:M_n-Mazur}, \ref{thm:M_n-cork}, and Proposition~\ref{prop:M_n-hyp}).

 To prove that these corks cannot admit relative trisections of genus less than $3$, we give a lower bound for trisection genus.

\begin{thm}\label{mainthm:lowerbound}
 Let $X$ be a compact, connected, oriented, smooth $4$-manifold with connected boundary, and let $\chi(X)$ denote the Euler characteristic.
 The trisection genus of $X$ is greater than $\chi(X)-2$.
 Moreover, for an integer $b\in\{1,2,3\}$, if $\partial{X}$ cannot admit a planar open book decomposition with $b$ binding components, then the trisection genus of $X$ is greater than $\chi(X)+b-2$.
\end{thm}

 As an application, we obtain the following corollary.

\begin{cor}\label{maincor:bdryOB}
 The minimal number of binding components of planar open book decompositions on each $\partial{M_n}$ is $4$.
\end{cor}

 Note that $\partial{M_1}$ is the boundary of the Akbulut cork.
 Theorem~\ref{mainthm:lowerbound} and Corollary~\ref{maincor:bdryOB} are proved by using the property that a relative trisection induces an open book decomposition on the boundary. 
 See section~\ref{sec:lowerbound} for details.

 We now turn our attention to trisections of exotic $4$-manifolds. 
 In \cite{LamMei20}, Lambert-Cole and Meier conjectured that trisection genus is additive under connected sum, that is, for any $4$-manifolds $X$ and $Y$, the trisection genus of $X\# Y$ is the sum of the trisection genera of $X$ and $Y$.
 They also showed that, if this is true, it follows that trisection genus is a homeomorphism invariant and there are no exotic $S^4$ or some $4$-manifolds (e.g., $\CC{P^2}$, $S^1\times S^3$, $S^2\times S^2$).
 So it is interesting to construct trisections for exotic $4$-manifolds.
 See \cite{BaySae17a}, \cite{MeiZup18}, \cite{CasOzb19}, and \cite{LamMei20}, for examples of such trisections.
 A related problem is to find trisection diagrams for exotic pairs (\cite[Problem~1.26]{AimPL}).
 In Section~\ref{sec:exotic}, we construct low genus relative trisection diagrams of an exotic pair of small $4$-manifolds (see Figures~\ref{fig:P-rtd} and \ref{fig:Q-rtd}). The diagrams give the following theorem.

\begin{thm}\label{mainthm:exotictris}
 There exist an exotic pair of simply-connected compact $4$-manifolds $P$ and $Q$ with $b_2 = 1$ such that they admit relative trisections of genus $4$ and $5$, respectively.
 In particular, the trisection genus of $Q$ is $4$, and the trisection genus of $P$ is either $4$ or $5$.
\end{thm}

 A natural question is whether $P$ can admit a genus $4$ relative trisection.
 However, we have not been able to find it.
 If one can show that there is no such relative trisection, it follows that the trisection genus for $4$-manifolds with boundary is not homeomorphism invariant.

\section{Preliminaries}\label{sec:prel}

\subsection{Notation and conventions}

 Throughout this paper, we assume that manifolds are compact, connected, oriented, and smooth unless otherwise stated. In addition, we will use the following notation.
\begin{itemize}
\item
 If two manifolds $X$ and $Y$ are orientation-preserving diffeomorphic to each other, then we write $X\cong Y$.
\item
 Let $\Sigma_{g,b}$ be a compact, connected, oriented surface of genus $g$ with $b$ boundary components.
\item
 For a manifold $X$ and a submanifold $A\subset X$, we denote a tubular neighborhood of $A$ in $X$ by $\nu(A;X)$.
\end{itemize}

\subsection{Relative trisections}

 We now introduce the definition of relative trisections rephrased by Castro, Gay, and Pinz\'{o}n-Caiced \cite{CasGayPin18_1}.
 A relative trisection is a decomposition of a $4$-manifold with connected boundary into three $4$-dimensional $1$-handlebodies.
 To describe how to glue these three pieces, we construct a model $Z_{k}$ of a genus $k$ $4$-dimensional $1$-handlebody.
 Let $g$, $k$, $p$, and $b$ be integers satisfying $g,k,p\geq0$, $b\geq1$, and $2p+b-1\leq k\leq g+p+b-1$.
 Note that the three integers $g-p$, $g-k+p+b-1$, and $k-2p-b+1$ are non-negative.
 Ultimately, we will define $Z_{k}$ as a boundary connected sum of two $4$-manifolds ${U}$ and ${V}$, so we start by defining these.

 First, we construct ${U}$ and give a decomposition of the boundary.
 Let $D$ be the third of the unit disk defined as follows.
\begin{equation*}
D:=\{ (r,\theta)\in \CC \mid r\in[0,1], \theta \in[-\pi/3,\pi/3]\}.
\end{equation*}
 Give a decomposition of the boundary as $\partial{D}=\partial^-D\cup \partial^0D\cup \partial^+D$, where
\begin{align*}
\partial^\pm D &:= \{ (r,\theta)\in D\mid r\in[0,1],\theta=\pm \pi/3\} \quad \textrm{and} \\
\partial^0D &:= \{ (r,\theta)\in D\mid r=1,\theta \in [-\pi/3,\pi/3] \}.
\end{align*}
 Let $P:=\Sigma_{p,b}$ and ${U} := D\times P$.
 We immediately see that ${U} \cong \natural^{2p+b-1}(S^1\times D^3)$.
 Decompose the boundary of ${U}$ as $\partial{{U}}=\partial^-{U} \cup \partial^0U_{p,b} \cup \partial^+{U}$, where
\begin{align*}
\partial^\pm {U} := \partial^\pm{D}\times P \quad \textrm{and} \quad \partial^0{{U}} := (\partial^0{D}\times P)\cup(D\times \partial{P}).
\end{align*}

 Next, we construct ${V}$ and give a decomposition of the boundary.
 Consider a $4$-dimensional solid torus $S^1\times D^3$, and decompose the boundary as follows.
\begin{equation*}
\partial(S^1\times D^3) = S^1\times(S^2_-\cup S^2_+) = \partial^-(S^1\times D^3)\cup \partial^+(S^1\times D^3),
\end{equation*}
 where $S_{\pm}^2$ are the northern and southern hemispheres of $\partial{D^3}$, and $\partial^\pm(S^1\times D^3):=S^1\times S_{\pm}^2$.
 This decomposition is the standard genus $1$ Heegaard splitting of $S^1\times S^2$.
 For $k-2p-b+1$ copies of $S^1\times D^3$ with the Heegaard splitting of the boundary, define $V_{k-2p-b+1}:= \natural^{k-2p-b+1}(S^1\times D^3)$, where the boundary connected sums are taken in neighborhoods of points in the Heegaard surfaces of copies of $\partial(S^1\times D^3)$ (see Figure~\ref{fig:connsum-S1xD3}).
 Let $\partial^\pm{V_{k-2p-b+1}}:=\natural^{k-2p-b+1}{\partial^\pm(S^1\times D^3)}$, then $\partial{V_{k-2p-b+1}}=\partial^-{V_{k-2p-b+1}}\cup \partial^+{V_{k-2p-b+1}}$.
 This is the standard genus $k-2p-b+1$ Heegaard splitting of $\#^{k-2p-b+1}(S^1\times S^2)$.
 We denote the result of stabilizing $g-k+p+b-1$ times by $\partial{{V}}=\partial^-{{V}}\cup \partial^+{{V}}$, which has genus $g-p$.
 Note that it is independent of the stabilizations (see Section~4 in \cite{Wal68}).
\begin{figure}[!htbp]
\centering
\includegraphics[scale=1.0]{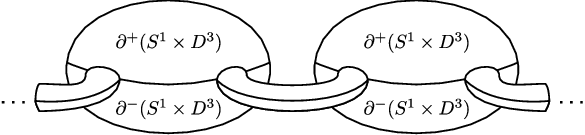}
\caption{An image of the boundary connected sums $V_{k-2p-b+1}= \natural^{k-2p-b+1}(S^1\times D^3)$.}
\label{fig:connsum-S1xD3}
\end{figure}

 Finally, we define $Z_{k} := {U}\natural {V}$, where the boundary connected sum is taken in neighborhoods of points in $\Int(\partial^-{U}\cap \partial^+{U})$ and $\partial^-{V}\cap \partial^+{V}$ (see Figure~\ref{fig:connsum-UV}).
 We see that $Z_{k}$ is diffeomorphic to $\natural^k(S^1\times D^3)$.
 Let $Y_{k} :=\partial{Z_{k}}$, and give a decomposition $Y_{k} = Y^-_{g,k;p,b}\cup Y^0_{g,k;p,b}\cup Y^+_{g,k;p,b}$, where
\begin{align*}
 Y^\pm_{g,k;p,b} := \partial^{\pm}{U}\natural \partial^{\pm}{{V}} \quad \textrm{and} \quad Y^0_{g,k;p,b} := \partial^0U.
\end{align*}
 See Sections~3 and 4 in \cite{CasGayPin18_1} for details of this model. Now we are ready to define a relative trisection.
\begin{figure}[!htbp]
\centering
\includegraphics[scale=1.0]{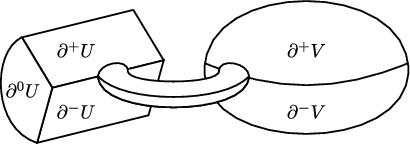}
\caption{An image of the boundary connected sum  $Z_{k}={U}\natural {V}$.}
\label{fig:connsum-UV}
\end{figure}

\begin{defi}[Castro--Gay--Pinz\'{o}n-Caiced {\cite[Definition~10]{CasGayPin18_1}}]\label{def:rt}
 Let $g$, $k$, $p$, and $b$ be integers satisfying $g,k,p\geq0$, $b\geq1$, and $2p+b-1\leq k\leq g+p+b-1$.
 Let $X$ be a compact, connected, oriented, smooth $4$-manifold with connected boundary.
 A decomposition $X=X_1\cup X_2\cup X_3$ is called a $(g,k;p,b)$-\textit{relative trisection} of $X$ if it satisfies the following conditions.
\begin{enumerate}
\item
 For each $i\in\{1,2,3\}$, there is a diffeomorphism $\phi_i:X_i\to Z_{k}$.
\item
 For each $i \in\{1,2,3\}$, taking indices mod $3$,
\begin{equation*}
\phi_i(X_i\cap X_{i\pm1})=Y^\mp_{g,k;p,b} \quad \textrm{and} \quad \phi_{i}(X_i\cap \partial{X})=Y^0_{g,k;p,b}.
\end{equation*}
\end{enumerate}
\end{defi}

 In this paper, we sometimes denote a relative trisection $X=X_1\cup X_2\cup X_3$ by $\calT$.
 The genus of the triple intersection surface $\Sigma := X_1\cap X_2\cap X_3$ is called the \textit{genus} of $\calT$.
 The \textit{trisection genus} of $X$ is the minimal integer $g$ such that $X$ admits a (relative) trisection of genus $g$.
 This is an invariant for smooth $4$-manifolds.

\begin{lem}[Castro--Gay--Pinz\'{o}n-Caiced {\cite[Lemma~11]{CasGayPin18_1}}]\label{lem:obd}
 A $(g,k;p,b)$-relative trisection of $X$ induces an open book decomposition on the boundary $\partial{X}$ with pages of genus $p$ with $b$ boundary components.
\end{lem}

 For a $(g,k;p,b)$-relative trisection $X=X_1\cup X_2\cup X_3$, each integer in the $4$-tuple $(g,k;p,b)$ has the following meaning.
 The integer $g$ is the genus of the triple intersection surface, $k$ is the $4$-dimensional genus of each sector $X_i$, $p$ is the genus of the page of the induced open book decomposition, and $b$ is the number of binding components.

\begin{prop}[Castro--Ozbagci {\cite[Corollary~2.10]{CasOzb19}}]\label{prop:Euler}
 Suppose that a $4$-manifold $X$ admits a $(g,k;p,b)$-relative trisection.
 Then the Euler characteristic $\chi(X)$ is equal to $g-3k+3p+2b-1$.
\end{prop}

 Next, we introduce the definition of a relative trisection diagram.

\begin{defi}[Castro--Gay--Pinz\'{o}n-Caiced{\cite[Definition~1]{CasGayPin18_1}}]
 Let $\Sigma$ and $\Sigma'$ be compact, connected, oriented surfaces.
 For $i\in \{1,\ldots,n\}$, let $\alpha^i$ and $\beta^i$ be families of $k$ pairwise disjoint simple closed curves on $\Sigma$ and $\Sigma'$, respectively.
 Two $n+1$-tuples $(\Sigma;\alpha^1,\ldots,\alpha^n)$ and $(\Sigma';\beta^1,\ldots,\beta^n)$ are \textit{diffeomorphism and handle slide equivalent}  if they are related by diffeomorphisms of $\Sigma$ and handle slides within each $\alpha^i$ (i.e., we are only allowed to slide curves from $\alpha^i$ over other curves from $\alpha^i$, but not over curves from $\alpha^j$ when $j\neq i$).
\end{defi}

\begin{defi}[Castro--Gay--Pinz\'{o}n-Caiced {\cite[Definition~2]{CasGayPin18_1}}]\label{def:rtd}
 Let $g$, $k$, $p$, and $b$ be integers satisfying $g,k,p\geq0$, $b\geq1$, and $2p+b-1\leq k\leq g+p+b-1$.
 Let $\Sigma$ be a surface diffeomorphic to $\Sigma_{g,b}$, and let $\alpha$, $\beta$, and $\gamma$ be families of $g-p$ pairwise disjoint simple closed curves on $\Sigma$. 
 A $4$-tuple $(\Sigma;\alpha,\beta,\gamma)$ is called a $(g,k;p,b)$-\textit{relative trisection diagram} if $(\Sigma;\alpha,\beta)$, $(\Sigma;\beta,\gamma)$, and $(\Sigma;\gamma,\alpha)$ are diffeomorphism and handle slide equivalent to the standard diagram $(\Sigma_{g,b};\delta,\epsilon)$ shown in Figure \ref{fig:std-rtd}, where the red curves are $\delta$ and the blue curves are $\epsilon$.
\begin{figure}[!htbp]
\centering
\includegraphics[scale=1.0]{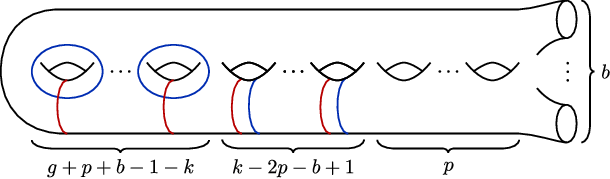}
\caption{The standard diagram $(\Sigma_{g,b};\delta,\epsilon)$ of type $(g,k;p,b)$.}
\label{fig:std-rtd}
\end{figure}
\end{defi}

 In this paper, we sometimes denote a relative trisection diagram by $\calD$.
 When considering a relative trisection diagram of the form $(\Sigma;\alpha,\beta,\gamma)$, we represent $\alpha$, $\beta$, and $\gamma$ curves by red, blue, and green curves, respectively (see Figure~\ref{fig:D_n-rtd} for example).
 A pair of black disks with a number indicates an attaching of a cylinder (see Figure~\ref{fig:blackhole}).
\begin{figure}[!htbp]
\centering
\includegraphics[scale=1.0]{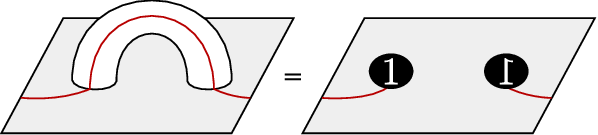}
\caption{A meaning of a pair of black disks with a number.}
\label{fig:blackhole}
\end{figure}

 The following theorem gives a natural correspondence between relative trisections and relative trisection diagrams.

\begin{thm}[Castro--Gay--Pinz\'{o}n-Caiced {\cite[Theorem~3]{CasGayPin18_1}}]\label{thm:rt-rtd}
 The following (i), (ii), and (iii) hold.
\begin{enumerate}
\item
 For any $(g,k;p,b)$-relative trisection diagram $(\Sigma;\alpha,\beta,\gamma)$, there exists a unique (up to diffeomorphism) trisected $4$-manifold $X=X_1\cup X_2\cup X_3$ satisfying the following conditions.
\begin{itemize}
\item
 $X_1\cap X_2\cap X_3\cong \Sigma$.
\item
 Under the above identification, each of $\alpha$, $\beta$, and $\gamma$ curves bound compressing disks of $X_1\cap X_2$, $X_2\cap X_3$, and $X_3\cap X_1$, respectively.
\end{itemize}
\item
 For any relative trisection $\calT$, there exists a relative trisection diagram $\calD$ such that $\calT$ is induced from $\calD$ by (i).
\item
 Let $\calD$ and $\calD'$ be relative trisection diagrams. If the relative trisections corresponding to these diagrams are diffeomorphic, then $\calD$ and $\calD'$ are diffeomorphism and handle slide equivalent.
\end{enumerate}
\end{thm}

 For a $4$-manifold $X$ and a relative trisection diagram $\calD$, if $X$ is diffeomorphic to the trisected $4$-manifold corresponding to $\calD$ by Theorem~\ref{thm:rt-rtd}, then we simply say that $\calD$ is a relative trisection diagram of $X$.

 There is a transition between handlebody diagrams and relative trisection diagrams.
 Castro, Gay, and Pinz\'{o}n-Caiced \cite{CasGayPin18_2} showed how to obtain a relative trisection from a handle decomposition of a $4$-manifold.
 Moreover, they gave an algorithm to construct relative trisection diagrams from handlebody diagrams.
 Conversely, Kim and Miller \cite{KimMil20} described a method for constructing handlebody diagrams from relative trisection diagrams.

\subsection{Corks}\label{subsec:Corks}

 A pair of smooth manifolds are said to be \textit{exotic} if they are homeomorphic but not diffeomorphic. Corks are used to construct exotic smooth structures of $4$-manifolds.

\begin{defi}
 Let $C$ be a compact, contractible, smooth $4$-manifold with boundary and $\tau:\partial{C}\to \partial{C}$ be a smooth involution on the boundary.
 The pair $(C,\tau)$ is called a \textit{cork}, if $\tau$ extends to a self-homeomorphism of $C$, but cannot extend to any self-diffeomorphism of $C$.
\end{defi}

 Let $X$ be a smooth $4$-manifold, and let $(C,\tau)$ be a cork.
 Suppose that $C$ is embedded in $X$.
 Let $X'$ be the $4$-manifold obtained from $X$ by removing $C$ and re-gluing it by $\tau$ (i.e., $X' := (X-C)\cup_{\tau}C$).
 Then we say that $X'$ is obtained from $X$ by a \textit{cork twist} along $(C,\tau)$.
 The $4$-manifolds $X$ and $X'$ are homeomorphic, but they may not be diffeomorphic.
 Conversely, it is known that any two simply-connected, closed, exotic $4$-manifolds are related by a cork twist (\cite{Mat96}, \cite{CurFreHsiSto96}). 
 We now introduce an example of corks.

\begin{thm}[Akbulut \cite{Akb91_1}]
 Let $W_1$ be the smooth $4$-manifold given by the handlebody diagram in Figure~\ref{fig:Ac-Kd}.
 Let $f_1:\partial{W_1}\to\partial{W_1}$ be the involution obtained by first surgering $S^1\times D^3$ embedded along the core of the $1$-handle to $D^2\times S^2$ in the interior of $W_1$, and then surgering $D^2\times S^2$ embedded along the core of the $2$-handle to $S^1\times D^3$ (i.e., replacing the ``dot'' and ``0'' in Figure~\ref{fig:Ac-Kd}).
 Then the pair $(W_1,f_1)$ is a cork.
\end{thm}

 This is the first example of a cork and is called the \textit{Akbulut cork}.
 Obviously the Akbulut cork is Mazur-type.
 Here a contractible smooth $4$-manifold $X$ is called \textit{Mazur-type}, if $X$ admits a handle decomposition consisting of one $0$-handle, one $1$-handle, and one $2$-handle.
\begin{figure}[!htbp]
\centering
\includegraphics[scale=1.0]{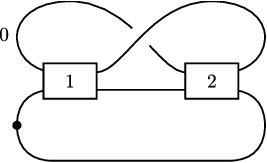}
\caption{A handlebody diagram of the Akbulut cork.}
\label{fig:Ac-Kd}
\end{figure}

\section{Genus $3$ relative trisections of corks}\label{sec:genus3}

 In this section, we construct an infinite family of corks with genus $3$ relative trisection. First, we give $(3,3;0,4)$-relative trisection diagrams.

\begin{lem}\label{lem:rtd-prf}
 For each positive integer $n$, let $\calD_n=(\Sigma;\alpha,\beta,\gamma)$ be the diagram shown in Figure~\ref{fig:D_n-rtd}, where $\Sigma$ is the genus $3$ surface with $4$ boundary components, and $\alpha$, $\beta$, and $\gamma$ are the families of three curves of red, blue, and green, respectively.
 Then $\calD_n$ is a $(3,3;0,4)$-relative trisection diagram.
\end{lem}
\begin{figure}[!htbp]
\centering
\includegraphics[scale=1.0]{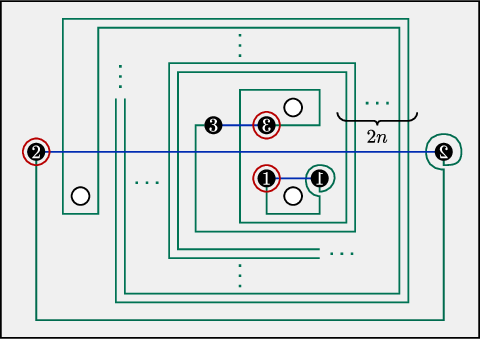}
\caption{$\calD_n=(\Sigma;\alpha,\beta,\gamma)$.}
\label{fig:D_n-rtd}
\end{figure}

\begin{proof}
 We prove that the triples $(\Sigma;\alpha,\beta)$, $(\Sigma;\beta,\gamma)$, and $(\Sigma;\gamma,\alpha)$ are diffeomorphism and handle slide equivalent to the standard diagram in Figure~\ref{fig:std3304}.
 To prove this, we introduce the four operations shown in Figures~\ref{fig:modif12} and \ref{fig:modif345}.
 The operation (iii) is obtained by two Dehn twists (see Figure~\ref{fig:modif3-prf}).
 Combining the operations (iii) and (i) shown in Figure~\ref{fig:modif4-prf}, we obtain the operation (iv), which exchanges a meridian and a longitude.
\begin{figure}[!htbp]
\centering
\includegraphics[scale=1.0]{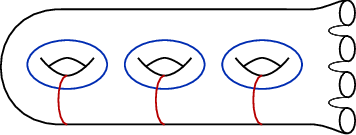}
\caption{The standard diagram of type $(3,3;0,4)$.}
\label{fig:std3304}
\end{figure}
\begin{figure}[!htbp]
\centering
\includegraphics[scale=1.0]{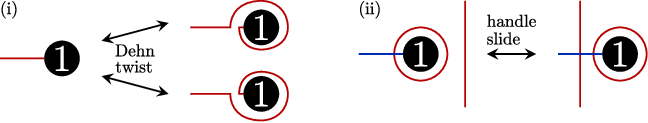}
\caption{The operations (i) and (ii).}
\label{fig:modif12}
\end{figure}
\begin{figure}[!htbp]
\centering
\includegraphics[scale=1.0]{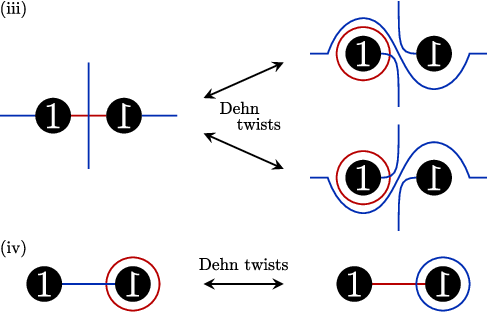}
\caption{The operations (iii) and (iv).}
\label{fig:modif345}
\end{figure}
\begin{figure}[!htbp]
\centering
\includegraphics[scale=1.0]{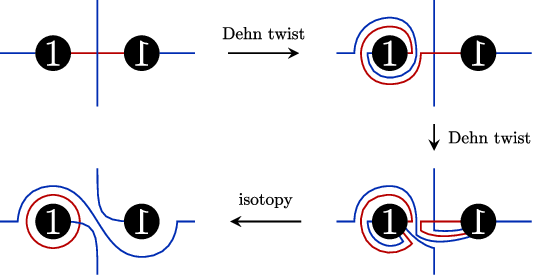}
\caption{A proof of the operation (iii).}
\label{fig:modif3-prf}
\end{figure}
\begin{figure}[!htbp]
\centering
\includegraphics[scale=1.0]{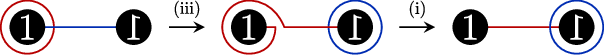}
\caption{A proof of the operation (iv).}
\label{fig:modif4-prf}
\end{figure}

 A proof for the case $n=1$ given by Figures~\ref{fig:sigma-ab}, \ref{fig:sigma-ca}, and \ref{fig:sigma-bc}.
 Note that the same operations can be performed for any $n\geq2$.
 The triple $(\Sigma;\alpha,\beta)$ in Figure~\ref{fig:sigma-ab} is already diffeomorphic to the standard diagram.
 $(\Sigma;\gamma,\alpha)$ can be made standard by only diffeomorphisms (see Figure~\ref{fig:sigma-ca}).
 The second and third diagrams are obtained by dragging the black disks along the marked $\gamma$ curves.
 Note that we can ignore the number of rotations of a $\gamma$ curve with respect to a black disk by the operation (i).
 Applying the operation (iv) to the last diagram, we obtain the standard diagram.
  $(\Sigma;\beta,\gamma)$ can be made standard by diffeomorphisms and handle slides shown in Figure~\ref{fig:sigma-bc}.
 The third diagram is obtained by dragging the black disk with blue circle along the marked $\gamma$ curve. In this process, when the black disk approaches a $\beta$ curve, it can pass through by the operation (ii).
\begin{figure}[!tbp]
\centering
\includegraphics[scale=1.0]{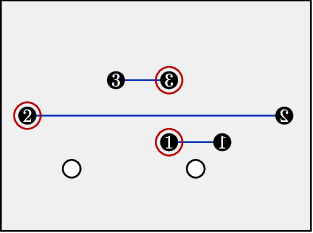}
\caption{$(\Sigma;\alpha,\beta)$.}
\label{fig:sigma-ab}
\end{figure}
\begin{figure}[!htbp]
\centering
\includegraphics[scale=1.0]{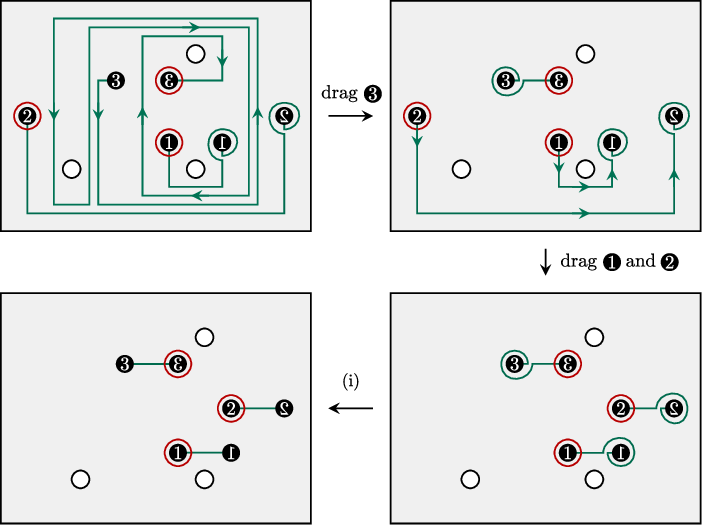}
\caption{Diffeomorphisms proving $(\Sigma;\gamma,\alpha)$ can be made standard.}
\label{fig:sigma-ca}
\end{figure}
\begin{figure}[!htbp]
\centering
\includegraphics[scale=1.0]{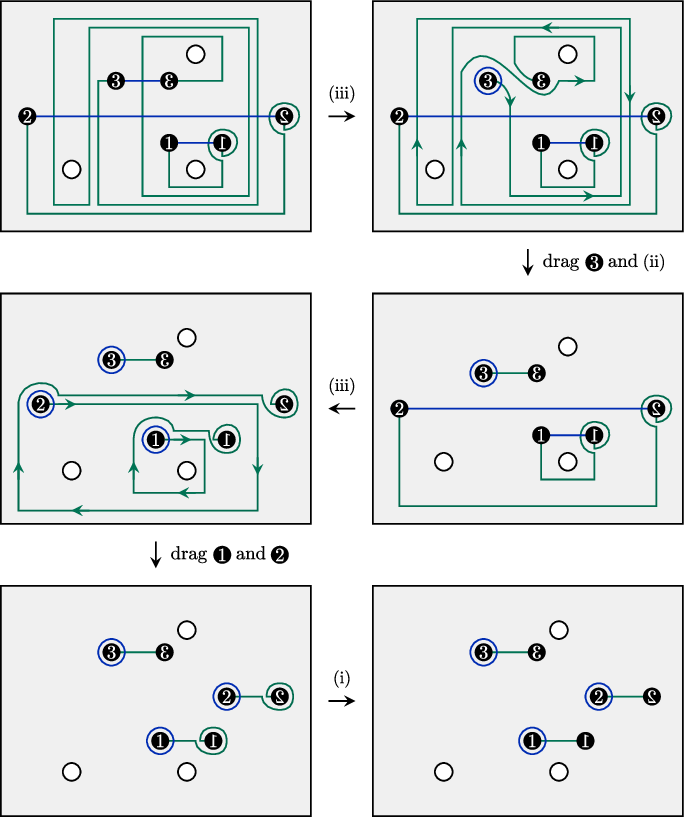}
\caption{Diffeomorphisms and handle slides proving $(\Sigma;\beta,\gamma)$ can be made standard.}
\label{fig:sigma-bc}
\end{figure}
\end{proof}

\begin{defi}\label{def:M_n}
 For each positive integer $n$, let $\calT_n$ be the $(3,3;0,4)$-relative trisection corresponding to $\calD_n$ by Theorem~\ref{thm:rt-rtd}, and let $M_n$ be the trisected $4$-manifold.
\end{defi}

\begin{lem}\label{lem:M_n-Kd}
Figure~\ref{fig:M_n-Kd} is a handlebody diagram of $M_n$.
\end{lem}

\begin{figure}[!htbp]
\centering
\includegraphics[scale=1.0]{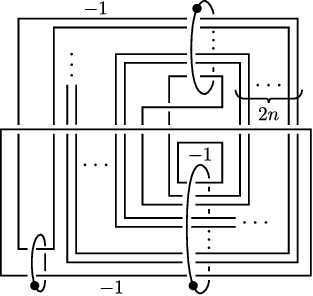}
\caption{The handlebody diagram of $M_n$ induced by $\calD_n$.}
\label{fig:M_n-Kd}
\end{figure}
\begin{figure}[!htbp]
\centering
\includegraphics[scale=1.0]{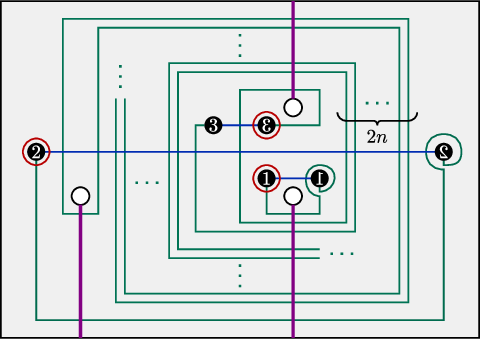}
\caption{$\calD_n$ and a cut system $\eta$.}
\label{fig:D_n-rtd-cutarcs}
\end{figure}

\begin{proof}
 By using the algorithm given by Kim and Miller \cite{KimMil20}, we obtain a handlebody diagram of $M_n$ as follows.
\begin{enumerate}
\item
 Perform diffeomorphisms of $\Sigma$ and handle slides so that $(\Sigma;\alpha,\beta)$ is standard, and then embed $(\Sigma;\alpha,\beta,\gamma)$ into $S^3$.
 The relative trisection diagram $\calD_n$ in Figure~\ref{fig:D_n-rtd} is already standard.
\item
Choose pairwise disjoint, properly embedded simple arcs $\eta\subset \Sigma$ that are disjoint from $\alpha$ and $\beta$ so that $\Sigma_{\alpha}-\nu(\eta;\Sigma) \cong \Sigma_{\beta}-\nu(\eta;\Sigma) \cong D^2$.
 Each of $\Sigma_\alpha$ and $\Sigma_\beta$ denotes the result of surgering $\Sigma$ along $\alpha$ and $\beta$ curves, respectively.
 We call such arcs $\eta$ a \textit{cut system} for $(\Sigma;\alpha,\beta)$.
 In our case, choose $\eta$ as the three purple arcs in Figure~\ref{fig:D_n-rtd-cutarcs}.
\item
 For each $\eta_i$, draw a dotted circle $C_i \subset \partial \nu(\Sigma;S^3)$ as shown in Figure~\ref{fig:choose-cutarcs} (i.e., under the identification $\nu(\Sigma;S^3)\cong \Sigma\times[-1,1]$, let $C_i:=\partial{(\eta_i\times[-1,1])}$).
\item
 Consider the $\gamma$ curves as attaching circles of $2$-handles with the surface framing of $\Sigma$.
\item
 Then $\{C_1,C_2,\ldots,C_{2p+b-1};\gamma_1,\gamma_2,\ldots,\gamma_{g-p}\}$ is a handlebody diagram of $M_n$.
\end{enumerate}

 Figure~\ref{fig:KMalg} shows this algorithm for the case $n=1$.
\end{proof}

\begin{figure}[!htbp]
\centering
\includegraphics[scale=1.0]{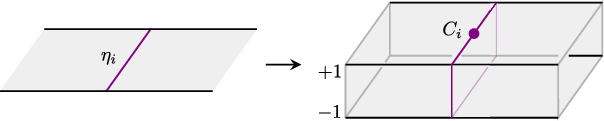}
\caption{How to obtain a dotted circle from a cut arc.}
\label{fig:choose-cutarcs}
\end{figure}
\begin{figure}[!htbp]
\centering
\includegraphics[scale=1.0]{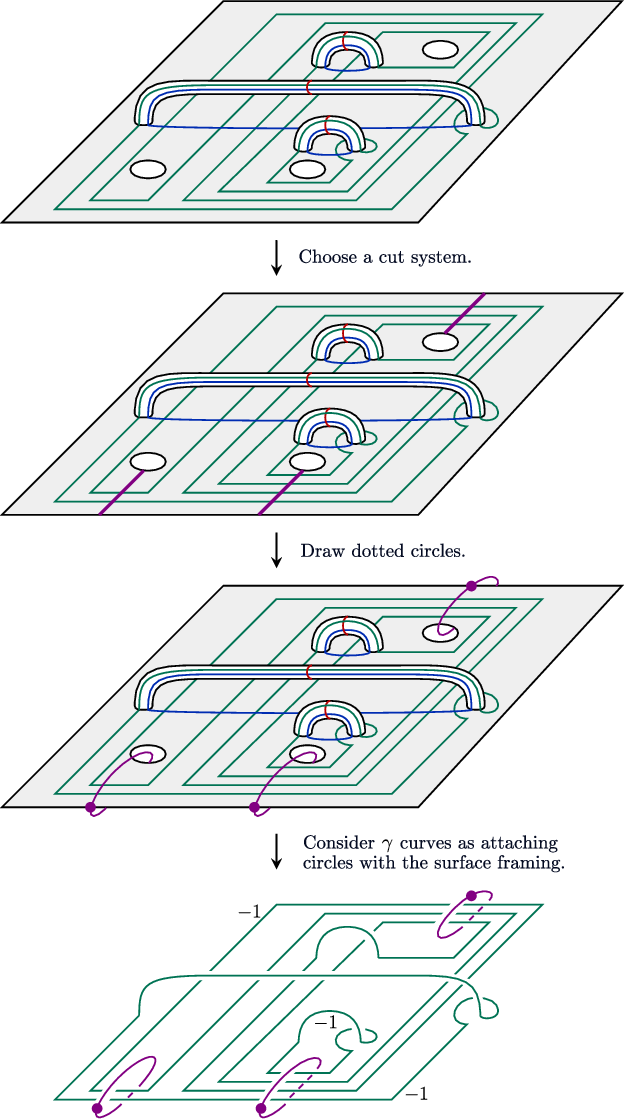}
\caption{The algorithm of \cite{KimMil20} for the case of $\calD_1$.}
\label{fig:KMalg}
\end{figure}

\begin{thm}\label{thm:M_n-Mazur}
 The family $\{M_n\}_{n\in \NN}$ satisfies the following conditions.
\begin{itemize}
\item Each $M_n$ is a Mazur-type $4$-manifold.
\item $M_1, M_2, \ldots$ are mutually non-homeomorphic.
\end{itemize}
\end{thm}

 The following proof is inspired by the work of Oba \cite{Oba15}.

\begin{proof}
 Performing the handle moves in Figure~\ref{fig:M_n-Kcalc}, we obtain the handlebody diagram of $M_n$ consisting of one $0$-handle, one $1$-handle, and one $2$-handle.
 By the Seifert--van Kampen theorem and the Mayer--Vietoris sequence, we see that $\pi_1(M_n)$ is trivial and $H_*(M_n) \cong H_*(\mathrm{pt.})$.
 Thus $M_n$ is contractible.
 
 Next, we calculate the Casson invariant $\lambda(\partial{M_n})$.
 Perform the handle moves in Figure~\ref{fig:bdryM_n-Kcalc} and, let $K_n$ be the knot of the last diagram.
 We see that $\partial{M_n}$ is obtained by Dehn surgery along $K_n$ with coefficient $1$.
 By the surgery formula for Casson invariants (see \cite{Sav02b}), the following holds.
\begin{equation*}
 \lambda(\partial{M_n})=\lambda(S^3+\frac{1}{1} K_n) = \lambda(S^3)+\frac{1}{2}\Delta''_{K_n}(1),
\end{equation*}
 where $\Delta''_{K_n}(1)$ is the second derivative of the Alexander polynomial $\Delta_{K_n}(t)$ at $t=1$.
 By Lemma~\ref{lem:Alex-K_n} described later, $\Delta''_{K_n}(1) = -2n(n+1)$.
 Hence we see that $\lambda(\partial{M_n}) = -n(n+1)$.
 Thus, $\partial{M_1}, \partial{M_2}, \ldots$ are mutually non-homeomorphic, so $M_1, M_2, \ldots$ are also mutually non-homeomorphic.
\begin{figure}[tbp]
\centering
\includegraphics[scale=1.0]{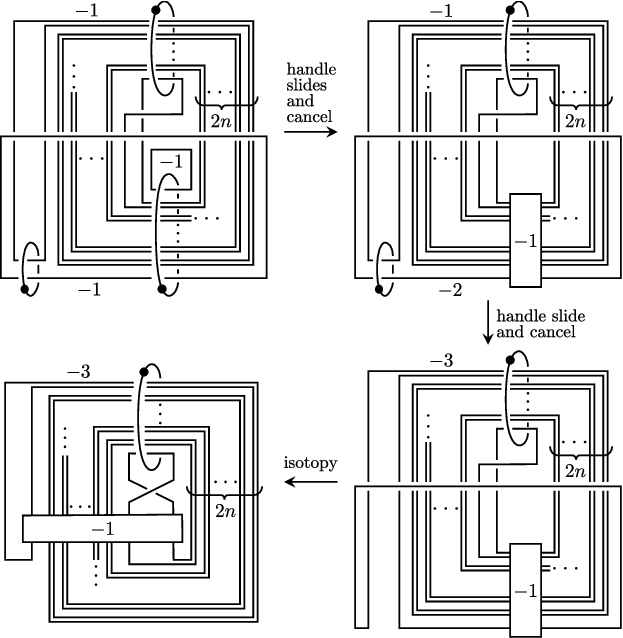}
\caption{Handle moves of $M_n$.}
\label{fig:M_n-Kcalc}
\end{figure}
\begin{figure}[!htbp]
\centering
\includegraphics[scale=1.0]{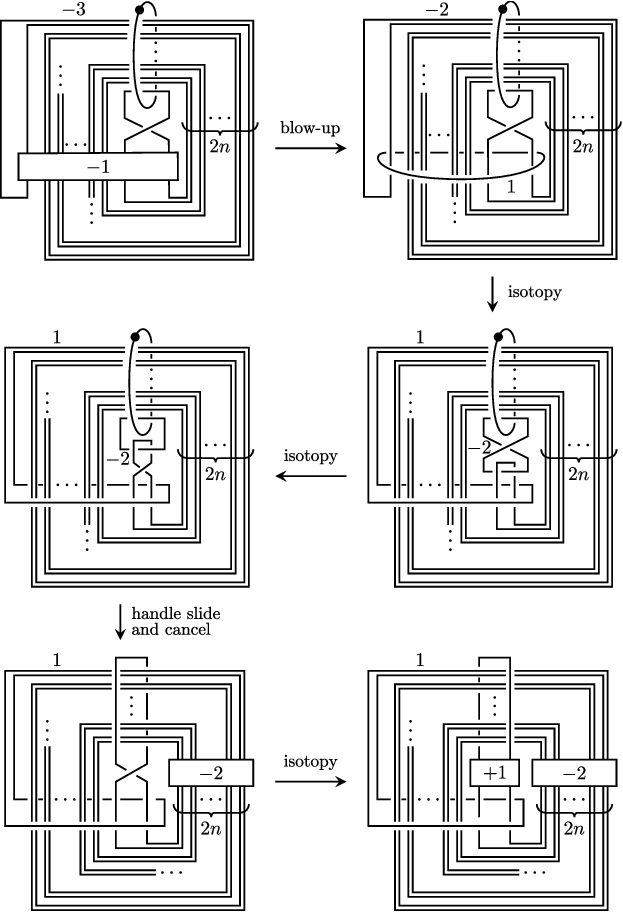}
\caption{Handle moves of $\partial{M_n}$.}
\label{fig:bdryM_n-Kcalc}
\end{figure}
\end{proof}

\begin{lem}\label{lem:Alex-K_n}
 The Alexander polynomial of the knot $K_n$ is given by
\begin{equation*}
 \Delta_{K_n}(t) = -n(n+1)t+(2n^2+2n+1)-n(n+1)t^{-1}.
\end{equation*}
\end{lem}

\begin{proof}
  Let $D_n\subset \RR^3$ be the immersed disk in Figure~\ref{fig:ribbondisk}.
 Since $\partial{D_n}=K_n$, the knot $K_n$ is ribbon and $D_n$ is its ribbon disk.
 By Fox and Milnor \cite[Theorem~2]{FoxMil66}, there exists a polynomial $f(t)$ with integer coefficients such that $\Delta_{K_n}(t)=f(t)f(t^{-1})$.
 We now calculate $f(t)$ by using the algorithm of Terasaka \cite{Ter59}.
 First orient the knot $K_n$, and decompose $D_n$ to two disks $A,C$ and $2n+1$ bands $B_1, B_2, \ldots, B_{2n+1}$ as shown in Figure~\ref{fig:ribbondisk}, where full twists of bands are represented by Figure~\ref{fig:fulltwist}.
 Then $f(t)$ is given by the following determinant of a $(2n+1)\times(2n+1)$ matrix.
\begin{equation*}
 f(t) = 
\begin{vmatrix}
-t^{\delta_1} &  & &  &           & t^{\delta^c_1}-1    \\
1        & -t^{\delta_2} &         &    &     \text{\huge{0}}   & t^{\delta^c_2}-1             \\
        & 1  & -t^{\delta_3} &  &     & t^{\delta^c_3}-1    \\
          &     &    &    \ddots   &  & \vdots \\
          & \text{\huge{0}}      &   &                  & -t^{\delta_{2n}}  & t^{\delta^c_{2n}}-1 \\
 &       &    &   &                 1          & -1 \\
\end{vmatrix}.
\end{equation*}
 Note that each of $\delta_i$ and $\delta^c_i$ is given as in Figure~\ref{fig:Terasaka_alg}.
 In the case of Figure~\ref{fig:ribbondisk}, $\delta_i=\delta^c_i=-1$ if $i$ is odd, and $\delta_i=1$, $\delta^c_i=0$ if $i$ is even.
 To compute this determinant, add all odd rows multiplied by $t$ and all even rows to the lowermost row. Then, all elements of the lowermost row become $0$ except for the rightmost one.
 By the cofactor expansion along the lowermost row, we obtain $f(t)=-(n+1)t+n$, so the formula of the lemma follows.
\begin{figure}[tbp]
\centering
\includegraphics[scale=1.0]{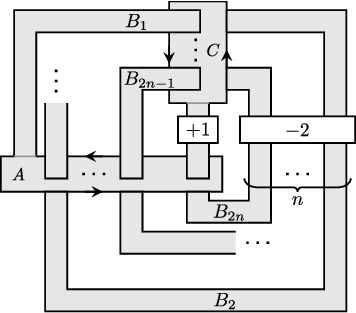}
\caption{A ribbon disk $D_n$ of $K_n$.}
\label{fig:ribbondisk}
\end{figure}
\begin{figure}[tbp]
\centering
\includegraphics[scale=0.9]{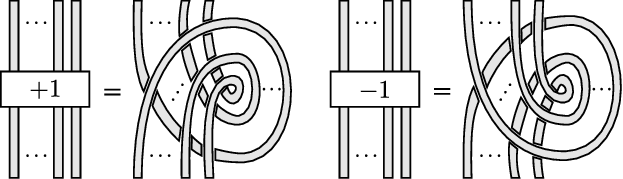}
\caption{Full twists of bands.}
\label{fig:fulltwist}
\end{figure}
\begin{figure}[tbp]
\centering
\includegraphics[scale=1.0]{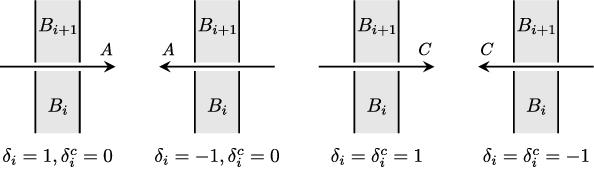}
\caption{How to determine $\delta_i$ and $\delta^c_i$.}
\label{fig:Terasaka_alg}
\end{figure}
\end{proof}

\begin{thm}\label{thm:M_n-cork}
Each $M_n$ is a cork and admits a Stein structure.
\end{thm}

\begin{proof}
 First, we prove that $M_n$ admits a Stein structure.
 Recall that the last diagram of Figure~\ref{fig:M_n-Kcalc} represents $M_n$.
 Perform the handle moves in Figure~\ref{fig:M_n-k.calc2}, where $(*)$ is the operation shown in Figure~\ref{fig:akmove} that was introduced in \cite{AkbKir79}. 
 For the last diagram of Figure~\ref{fig:M_n-k.calc2}, converting the $1$-handle notation, we obtain the Legendrian knot diagram in Figure~\ref{fig:M_n-Legendre}.
 Since the Thurston--Bennequin number is $1$, $M_n$ admits a Stein structure (\cite[Proposition~2.3]{Gom98}).
\begin{figure}[tbp]
\centering
\includegraphics[scale=1.0]{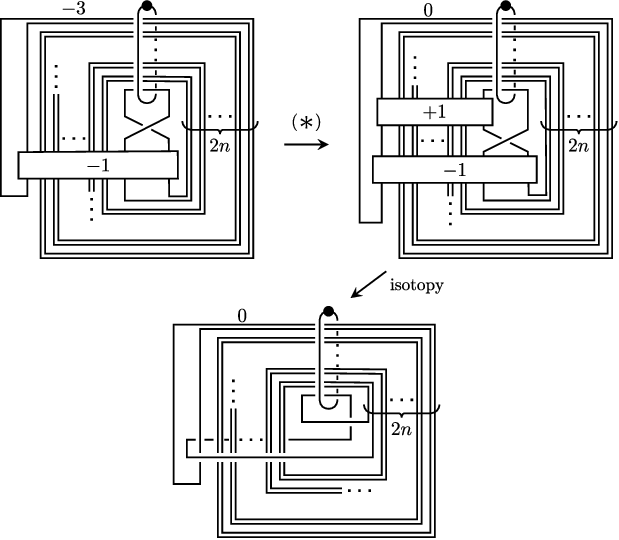}
\caption{Handle moves of $M_n$.}
\label{fig:M_n-k.calc2}
\end{figure}
\begin{figure}[tbp]
\centering
\includegraphics[scale=1.05]{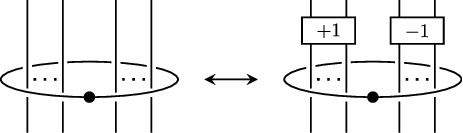}
\caption{The operation $(*)$.}
\label{fig:akmove}
\end{figure}
\begin{figure}[!tbp]
\centering
\includegraphics[scale=1.0]{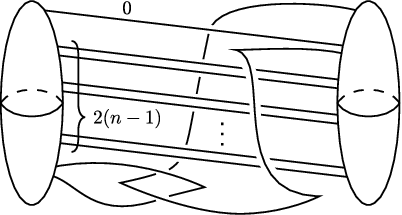}
\caption{A Legendrian knot diagram of $M_n$.}
\label{fig:M_n-Legendre}
\end{figure}

 Next, we prove that $M_n$ admits a cork structure.
 By the isotopies in Figure~\ref{fig:M_n-Kcalc3}, we obtain the diagram consisting of the $2$-component symmetric link.
 Note that the second isotopy is obtained by repeating the operation of Figure~\ref{fig:M_n-induc} $n-1$ times.
 Let $\tau_n:\partial{M_n}\to\partial{M_n}$ be the involution induced by $180^\circ$ rotation about the horizontal axis.
 Recall that $M_n$ is contractible and $\partial{M_n}$ is a homology $3$-sphere.
 By Boyer's theorem \cite{Boy86}, $\tau_n$ extends to a self-homeomorphism of $M_n$.
 Since $M_n$ is Stein, we can use Corollary~2.1 in \cite{AkbMat97} to prove that $\tau_n$ cannot extend to a self-diffeomorphism of $M_n$.
 For details of this argument, see the proof of Theorem~3.1 in \cite{AkbMat97}.
\end{proof}

\begin{figure}[!tbp]
\centering
\includegraphics[scale=1.0]{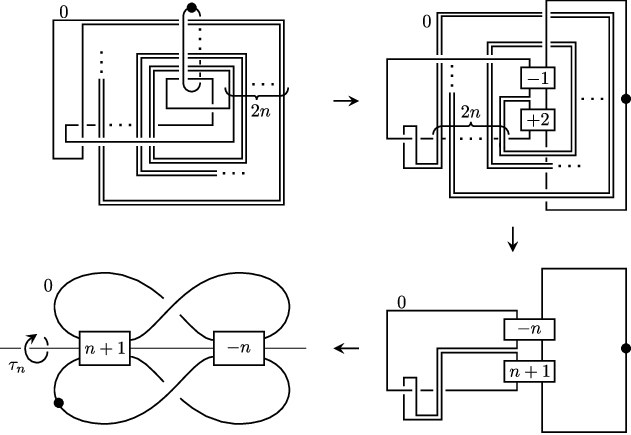}
\caption{Isotopies of $M_n$.}
\label{fig:M_n-Kcalc3}
\end{figure}
\begin{figure}[!tbp]
\centering
\includegraphics[scale=1.0]{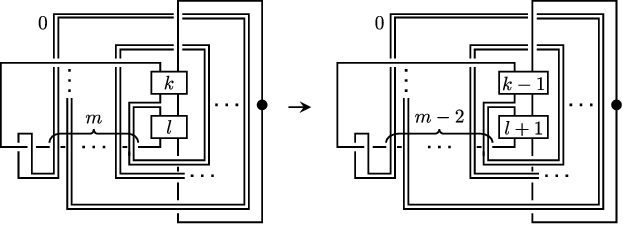}
\caption{Isotopies of $M_n$.}
\label{fig:M_n-induc}
\end{figure}

 We can easily check that the last diagram in Figure~\ref{fig:M_n-Kcalc3} for $n=1$ is isotopic to Figure~\ref{fig:Ac-Kd}, so $M_1$ is diffeomorphic to the Akbulut cork $W_1$.

\begin{rem}
 The $4$-manifolds $M_1,M_2,\ldots$ are already known corks.
 These were discovered by Dai, Hedden, and Mallick \cite[Theorem~1.12]{DaiHedMal20a}.
 They proved that $M_n$ admits a cork structure by using Heegaard Floer homology.
\end{rem}

\begin{prop}\label{prop:M_n-hyp}
 Each $\partial{M_n}$ is a hyperbolic $3$-manifold.
\end{prop}

\begin{proof}
 Performing the handle moves shown in Figure~\ref{fig:bdryM_n-Kcalc2}, we see that $\partial{M_n}$ is homeomorphic to the $3$-manifold obtained by Dehn surgery along the $2$-bridge link $L_{[2,-2,2]}$ with coefficients $\{2-\frac{1}{n+1}, 2+\frac{1}{n} \}$.
 For the notation $L_{[a_1,a_2,\ldots,a_k]}$, see Figures~1 and 2 in \cite{IchJonMas19a}.
 Since the type of the $2$-bridge link $L_{[2,-2,2]}$ is $(5, 12)$, so it is not $(2,n)$-torus link.
 By a result of Menasco \cite[Corollary~2]{Men84}, if a link $L$ is non-split, prime, alternating, and non-torus, then $L$ is hyperbolic.
 Thus $L_{[2,-2,2]}$ is a hyperbolic link.

 We prove that our surgeries are not exceptional for any $n\in\NN$.
 An \textit{exceptional surgery} is a Dehn surgery along a hyperbolic link with coefficients such that the resulting $3$-manifold is non-hyperbolic.
 In particular, an exceptional surgery along a link $L$ with coefficients is called \textit{complete}, if for any non-empty sublink $L'$, the $3$-manifold obtained from $S^3-\nu(L';S^3)$ by Dehn surgery along $L-L'$ is hyperbolic.
 Ichihara, Jong, and Masai \cite{IchJonMas19a} gave a complete list of hyperbolic $2$-bridge links that can admit complete exceptional surgeries. They also listed candidates of surgery coefficients of them. In addition, Ichihara \cite{Ich12} classified exceptional surgeries along components of hyperbolic $2$-bridge links.

 We verify that our surgeries are not included in these lists.
 It is known that two non-trivial $2$-bridge links of types $(p, q)$ and $(p', q')$ are isotopic if and only if $q=q'$ and either $p\equiv p'$ or $pp'\equiv1 \pmod q$ (see Section~2.1 of \cite{Kaw96b}).
 By using this, we see that $L_{[2,-2,2]}$ is not included in the list of \cite{Ich12}.
 However, among the $2$-bridge links in the list of \cite{IchJonMas19a}, only $L_{[3,2,3]}$ is isotopic to $L_{[2,-2,2]}$, which has $11$ candidate coefficients that would be a complete exceptional surgery.
 The first homology groups of the $3$-manifolds obtained by the surgeries with these coefficients are non-trivial except for one of the coefficients $\{ -3,-1 \}$.
 We calculate the Casson invariant of this exception.
 By blowing down a $-1$ framed circle shown in Figure~\ref{fig:5_2}, it can be represented by the mirror image of the knot $5_2$ with coefficient $1$.
 It is known that the Alexander polynomial of the knot $5_2$ is $2t-3+2t^{-1}$ (see \cite{Rol76b}).
 Hence the Casson invariant is $\frac{1}{2}\Delta''_{5_2}(1)=2$.
 Thus, it could be homeomorphic to the boundary of the Akbulut cork $\partial{M_1}$, since $\lambda(\partial{M_n})=-n(n+1)$.
 However, it is known to be hyperbolic (\cite[Theorem~1.2]{KarObaUki17}).
\begin{figure}[!htbp]
\centering
\includegraphics[scale=1.0]{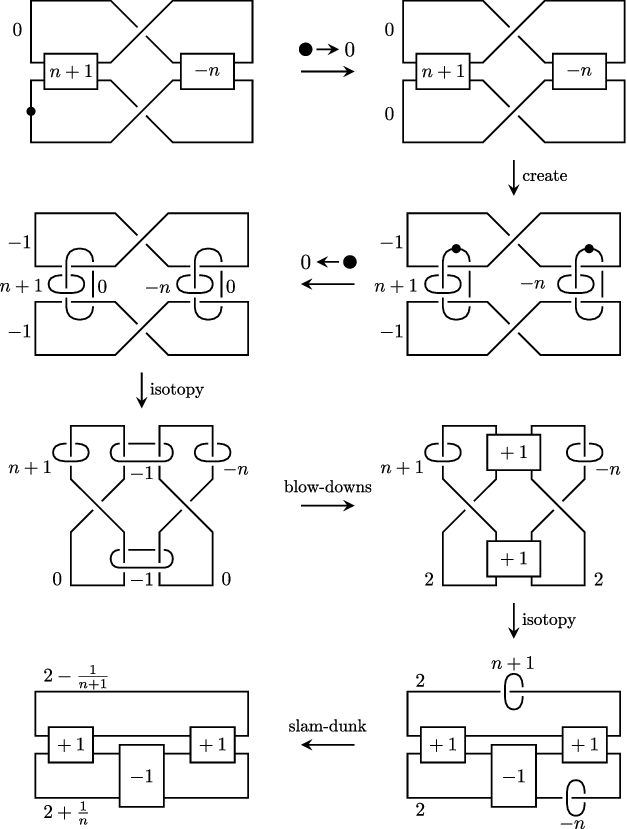}
\caption{Handle moves of $\partial{M_n}$.}
\label{fig:bdryM_n-Kcalc2}
\end{figure}
\begin{figure}[!htbp]
\centering
\includegraphics[scale=1.0]{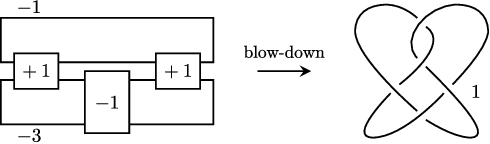}
\caption{A blowing down of a $-1$ framed circle.}
\label{fig:5_2}
\end{figure}
\end{proof}

\begin{rem}\label{rem:Thurston}
 Without using the result of the preprint \cite{IchJonMas19a}, we see that there are infinitely many hyperbolic $3$-manifolds contained in $\{\partial{M_n}\}_{n\in\NN}$.
 By Thurston's hyperbolic Dehn surgery theorem (\cite[Theorem~5.8.2]{Thu02b}), there exist finite subsets $E_1$ and $E_2$ of $\QQ$ such that, for any rational numbers $r_1\notin E_1$ and $r_2\notin E_2$, the $3$-manifold obtained by surgery along $L_{[2,-2,2]}$ with coefficients $\{ r_1,r_2 \}$ is hyperbolic.
 Thus, $\partial{M_n}$ is hyperbolic for any positive integer $n$ satisfying $2-\frac{1}{n+1}\notin E_1$ and $2+\frac{1}{n}\notin E_2$, and the set of such $n$ is infinite.
\end{rem}

\section{A lower bound for the trisection genus of a $4$-manifold with boundary}\label{sec:lowerbound}

 In this section, we prove that the trisection genus of $M_n$ is greater than $2$.
 The next lemma narrows down the $4$-tuple $(g,k;p,b)$ of relative trisections.

\begin{lem}\label{lem:rtd-Euler}
 Let $X$ be a compact, connected, oriented, smooth $4$-manifold with connected boundary.
 Suppose that $X$ admits a $(g,k;p,b)$-relative trisection.
 Let $A(g,k;p,b)$ be the number of pairs $(\delta_i,\epsilon_i)$ such that $\delta_i$ and $\epsilon_i$ intersect at one point in Figure~\ref{fig:std-rtd} (i.e., $A(g,k;p,b) := g+p+b-1-k$).
 Then $g,k,p,b$, and $A(g,k;p,b)$ satisfy the following conditions.
\begin{itemize}
\item
 $g\geq\chi(X)-1$.
\item
 $0\leq p \leq \mathrm{min}\{ \frac{g+1-\chi(x)}{3}, g\}$.
\item
 $\frac{2g-1+\chi(X)}{3} \leq A(g,k;p,b) \leq g-p$.
\item
 $k=1-\chi(X)-g+p+2A(g,k;p,b)$.
\item
 $b=3A(g,k;p,b)-2g+2-\chi(X)$.
\end{itemize}
\end{lem}

\begin{proof}
 The fourth and fifth conditions are obtained by combining the definition of $A(g,k;p,b)$ and the formula of Proposition~\ref{prop:Euler}.
 Since $b\geq1$, we see that $\frac{2g-1+\chi(X)}{3} \leq A(g,k;p,b)$, and Figure~\ref{fig:std-rtd} shows that $A(g,k;p,b) \leq g-p$.
 Hence the third condition holds.
 It also follows that $\frac{2g-1+\chi(X)}{3} \leq g-p$, so we obtain $p \leq \frac{g+1-\chi(x)}{3}$.
 Since $0\leq p \leq g$ (see Definition~\ref{def:rt}), the first and second conditions hold.
\end{proof}

\begin{proof}[Proof of Theorem~\ref{mainthm:lowerbound}]
 Assume that the trisection genus of $X$ is less than $\chi(X)+b-1$.
 Then there exist integers $k$, $p$, and $b'$ such that $X$ admits a $(\chi(X)+b-2, k; p, b')$-relative trisection.
 Using Lemma~\ref{lem:rtd-Euler}, we easily see that $k\leq b-1$, $p=0$, and $b'\leq b$.
 So $\partial{X}$ admits a planar open book decomposition with $b$ binding components.
 This is a contradiction.
\end{proof}

 Let $M$ be a closed $3$-manifold that admits a planar open book decomposition with binding number $b$.
 It is known that $M$ is the $3$-sphere if $b=1$, and $M$ is a lens space if $b=2$ (see \cite{EtnOzb08}).
 In addition, if $b=3$, $M$ is a Seifert fiber space or the connected sum of two lens spaces (see \cite{Ari08}).
 Combining these facts with Lemma~\ref{lem:rtd-Euler} and Theorem~\ref{mainthm:lowerbound}, we obtain the following corollary.

\begin{cor}\label{cor:lowerbound}
 For a $4$-manifold $X$ with boundary, if $\partial{X}$ is irreducible and not Seifert fibered, then the trisection genus of $X$ is greater than $\chi(X)+1$.
\end{cor}

\begin{rem}
 Theorem~\ref{mainthm:lowerbound} also holds for trisection genus of unbalanced relative trisections.
 A $(g,k_1,k_2,k_3;p,b)$-unbalanced relative trisection induces an open book decomposition with page $\Sigma_{p,b}$.
  In Lemma~\ref{lem:rtd-Euler}, by defining $A_i(g,k_i;p,b) := g+p+b-1-k_i$ for $i\in\{1,2,3\}$, we obtain the same result.
\end{rem}

\begin{proof}[Proof of Theorem~\ref{mainthm:corks}]
 By Proposition~\ref{prop:M_n-hyp}, each $\partial{M_n}$ is hyperbolic.
 It is known that a hyperbolic $3$-manifold is irreducible and not Seifert fibered (see \cite{Mil62} and \cite{Thu02b}), so we can use Corollary~\ref{cor:lowerbound}.
 Since the Euler characteristic of a contractible $4$-manifold is $1$, the trisection genus of $M_n$ is greater than $2$.
 On the other hand, $\calT_n$ is a genus $3$ relative trisection of $M_n$ (see Definition~\ref{def:M_n}).
 Therefore, we conclude that the trisection genus of $M_n$ is $3$ for any $n\in\NN$.
\end{proof}

\begin{rem} 
 For the proof of the hyperbolicity of $\partial{M_n}$, we used the result of the preprint \cite{IchJonMas19a}.
 Without using this, we can show that there are infinitely many corks with trisection genus $3$ by Remark~\ref{rem:Thurston}.
\end{rem}

\begin{proof}[Proof of Corollary~\ref{maincor:bdryOB}]
 Recall that $\calT_n$ is a $(3,3;0,4)$-relative trisection of $M_n$.
 By Lemma~\ref{lem:obd}, $\partial{M_n}$ admits an open book decomposition with page $\Sigma_{0,4}$.
 If $\partial{M_n}$ admits a planar open book decomposition of the number of binding components less than $4$, then it also contradicts Proposition~\ref{prop:M_n-hyp}.
\end{proof}

 Finally, we introduce a Heegaard splitting of $\partial{M_n}$ induced by our relative trisections.

\begin{lem}\label{prop:rt-Hspl}
 Let $X$ be a compact, connected, oriented, smooth $4$-manifold with connected boundary.
 If $X$ admits a $(g,k;p,b)$-relative trisection, then $\partial{X}$ admits a genus $2p+b-1$ Heegaard splitting.
\end{lem}

\begin{proof}
 By Lemma~\ref{lem:obd}, there exists an open book decomposition $(B,\pi)$ on $\partial{X}$ with page $\Sigma_{p,b}$, where $B$ is a binding and $\pi:\partial{X}-B \to S^1$ is a fibration.
 Under the identification $S^1\cong[0,1]/\sim$, we obtain the genus $2p+b-1$ Heegaard splitting $\partial{X}=(\pi^{-1}([0,1/2])\cup B)\cup(\pi^{-1}([1/2,1])\cup B)$.
\end{proof}

 By Definition~\ref{def:M_n} and Lemma~\ref{prop:rt-Hspl}, we see that each $\partial{M_n}$ admits a genus $3$ Heegaard splitting.

\section{Relative trisections of an exotic pair of $4$-manifolds}\label{sec:exotic}

 Recall that $M_1$ is diffeomorphic to the Akbulut cork $W_1$.
 Thus the trisection genus of $W_1$ is $3$, that is, Theorem~\ref{mainthm:Ac} holds.
 The diagram $\calD_1$ in Figure~\ref{fig:Ac-rtd} is a $(3,3;0,4)$-relative trisection diagram of $W_1$.
\begin{figure}[!htbp]
\centering
\includegraphics[scale=1.0]{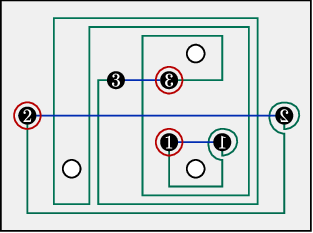}
\caption{A $(3,3;0,4)$-relative trisection diagram $\calD_1$ of the Akbulut cork $W_1$.}
\label{fig:Ac-rtd}
\end{figure}

 Let $P$ and $Q$ be the $4$-manifolds given by Figures~\ref{fig:P-Kd} and \ref{fig:Q-Kd}, respectively.
 They are exotic and related by a cork twist along $W_1$ (see subsection~9.1 in \cite{AkbYas13}).
 To prove Theorem~\ref{mainthm:exotictris}, we construct relative trisections of $P$ and $Q$.

\begin{figure}[!htbp]
	\begin{tabular}{cc}
		\begin{minipage}[t]{0.45\hsize}
		\centering
		\includegraphics[scale=1.0]{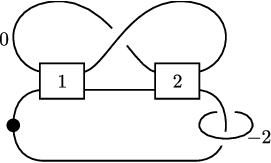}
		\caption{$P$}
		\label{fig:P-Kd}
		\end{minipage} &
		\begin{minipage}[t]{0.45\hsize}
		\centering
		\includegraphics[scale=1.0]{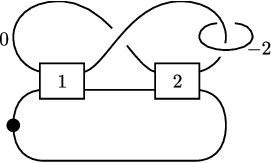}
		\caption{$Q$}
		\label{fig:Q-Kd}
		\end{minipage}
	\end{tabular}
\end{figure}

\begin{proof}[Proof of Theorem~\ref{mainthm:exotictris}]
 Let $\calD_P$ and $\calD_Q$ be the diagrams in Figures~\ref{fig:P-rtd} and \ref{fig:Q-rtd}, respectively.
 First, we prove that $\calD_P$ and $\calD_Q$ are relative trisection diagrams.
 As in the proof of Lemma~\ref{lem:rtd-prf}, we verify that $(\Sigma,\alpha,\beta)$, $(\Sigma,\beta,\gamma)$, and $(\Sigma,\gamma,\alpha)$ are diffeomorphism and handle slide equivalent to the standard diagram in Figure~\ref{fig:std-rtd}.
 The cases of $(\Sigma,\alpha,\beta)$ and $(\Sigma,\gamma,\alpha)$ are easy.
 The proofs for the case of $(\Sigma;\beta,\gamma)$ are shown in Figures~\ref{fig:P-sigma-ca} and \ref{fig:Q-sigma-ca}.
\begin{figure}[!htbp]
\centering
\includegraphics[scale=1.0]{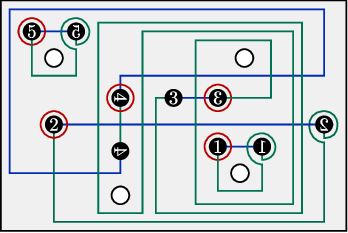}
\caption{A $(5,4;0,5)$-relative trisection diagram $\calD_P$.}
\label{fig:P-rtd}
\end{figure}
\begin{figure}[!htbp]
\centering
\includegraphics[scale=1.0]{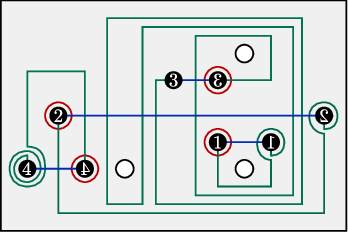}
\caption{A $(4,3;0,4)$-relative trisection diagram $\calD_Q$.}
\label{fig:Q-rtd}
\end{figure}

 Next, we show that $\calD_P$ and $\calD_Q$ are relative trisection diagrams of the $4$-manifolds $P$ and $Q$, respectively.
 Use the algorithm of \cite{KimMil20} to obtain the handlebody diagrams corresponding to these relative trisection diagrams, and perform the handle moves (see Figures~\ref{fig:P-KMalg} and \ref{fig:Q-KMalg}).
 We then see that they coincide with $P$ and $Q$ by the handle moves in Figure~\ref{fig:PQ-Kcalc}.
 For the last isotopy, see Figures~\ref{fig:M_n-k.calc2} and \ref{fig:M_n-Kcalc3} for the case $n=1$.

 We now consider lower bounds for the trisection genera of $P$ and $Q$.
 For the handlebody diagram of $Q$ in Figure~\ref{fig:Q-Kd}, using the slam-dunk move, we see that $\partial{Q}$ is obtained by Dehn surgery along the Mazur link (see the left diagram of Figure~\ref{fig:bdryQ}).
 To show that the manifold is hyperbolic, we use the techniques of Yamada \cite{Yam18}, he gave a complete list of exceptional integral surgeries along the Mazur link.
 The two $3$-manifolds given by Figure~\ref{fig:bdryQ} are diffeomorphic.
 The link of the right diagram is called the minimally twisted $4$-chain link.
 Exceptional surgeries along this link are classified by Martelli, Petronio and Roukema \cite{MarPetRou14}.
 By following the same argument as in subsection~3.3 of \cite{Yam18}, we see that our Dehn surgery is not exceptional, so $\partial{Q}$ is hyperbolic.
 Since $P$ and $Q$ are homeomorphic and $\chi(Q)=2$, the trisection genera are greater than $3$ by Theorem~\ref{mainthm:lowerbound}.
 Thus the trisection genus of $Q$ is $4$, and the trisection genus of $P$ is either $4$ or $5$.
\end{proof}

\begin{figure}[!htbp]
\centering
\includegraphics[scale=1.0]{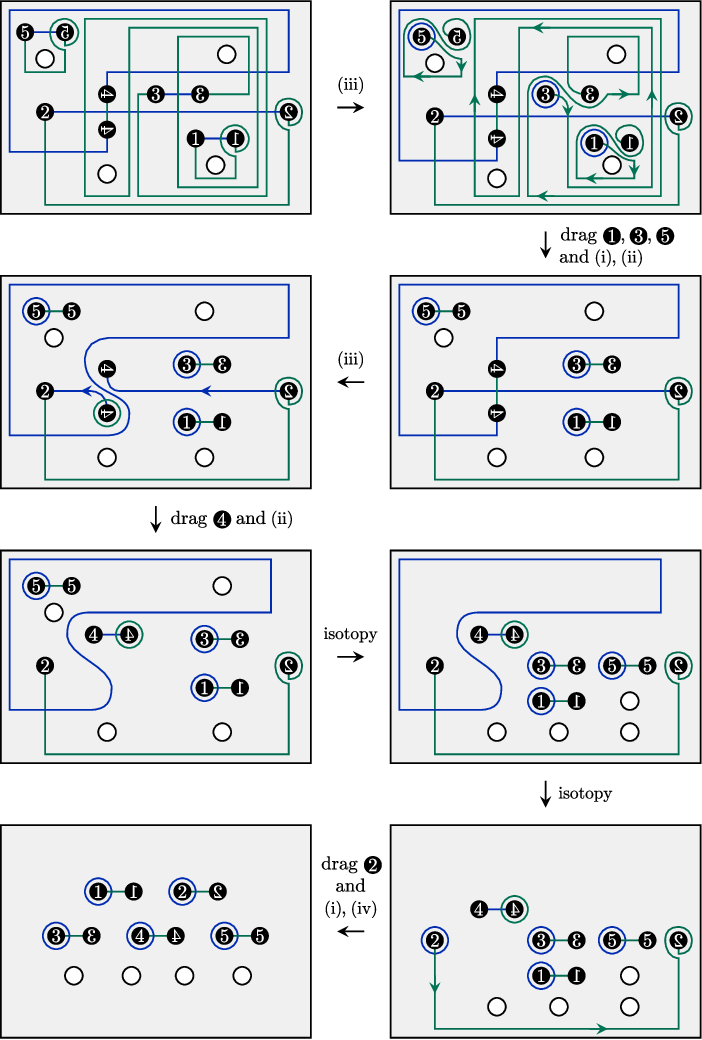}
\caption{Diffeomorphisms and handle slides proving $(\Sigma;\beta,\gamma)$ of $\calD_P$ can be made standard.}
\label{fig:P-sigma-ca}
\end{figure}
\begin{figure}[!htbp]
\centering
\includegraphics[scale=1.0]{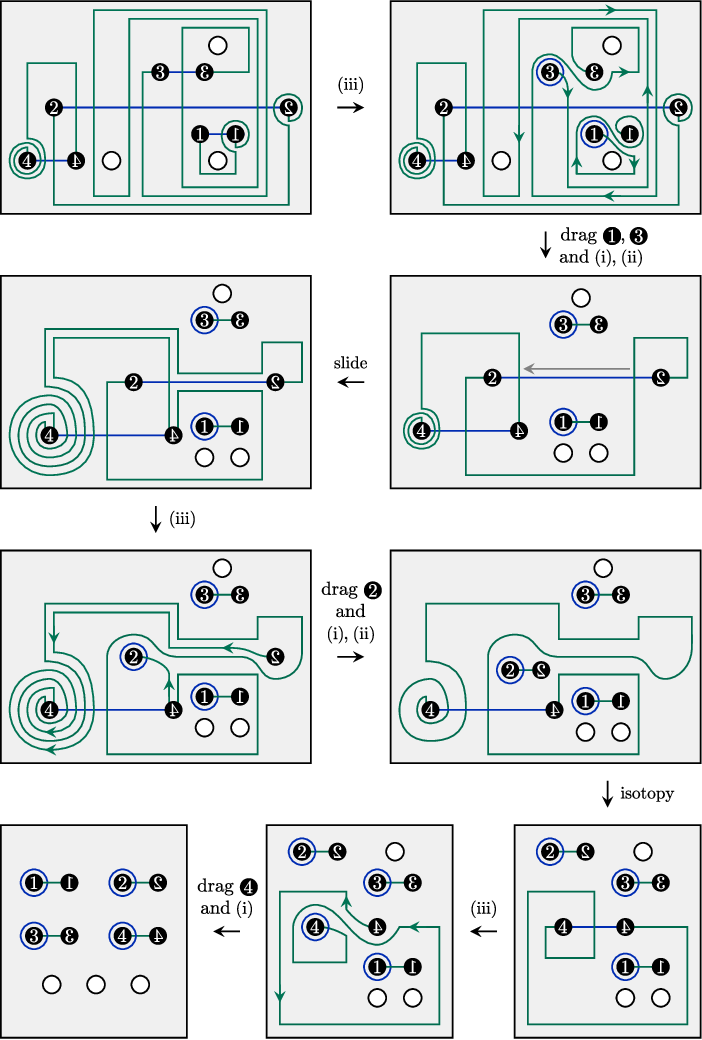}
\caption{Diffeomorphisms and handle slides proving $(\Sigma;\beta,\gamma)$ of $\calD_Q$ can be made standard.}
\label{fig:Q-sigma-ca}
\end{figure}
\begin{figure}[!tbp]
\centering
\includegraphics[scale=1.0]{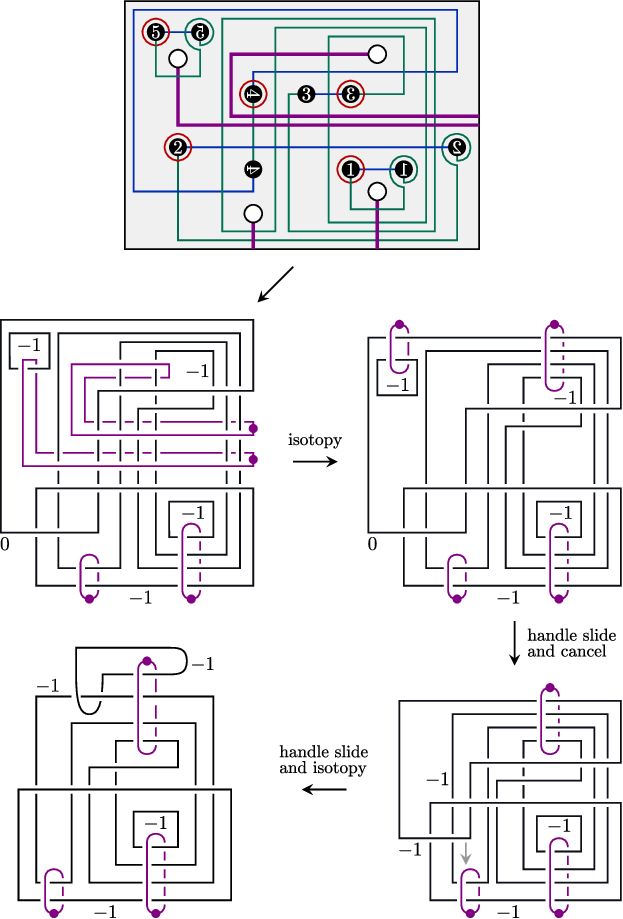}
\caption{$\calD_P$ with a cut system and handle moves of the induced $4$-manifold.}
\label{fig:P-KMalg}
\end{figure}
\begin{figure}[!htbp]
\centering
\includegraphics[scale=1.0]{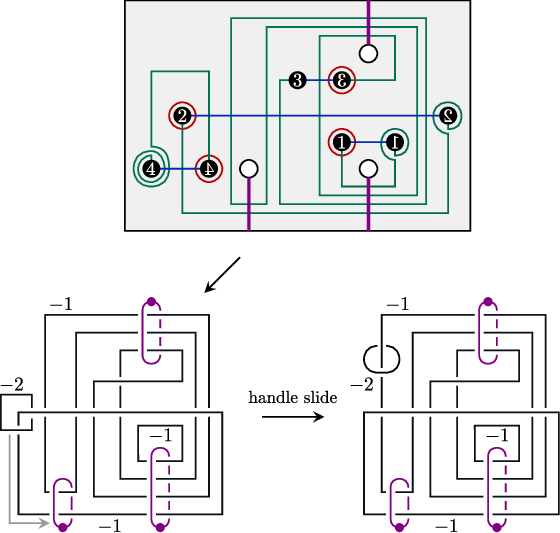}
\caption{$\calD_Q$ with a cut system and handle moves of the induced $4$-manifold.}
\label{fig:Q-KMalg}
\end{figure}
\begin{figure}[!htbp]
\centering
\includegraphics[scale=1.0]{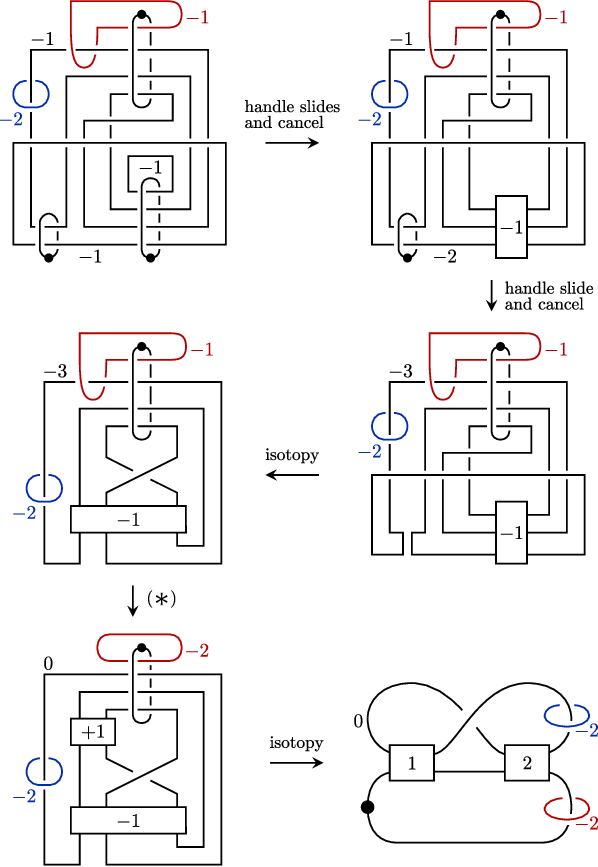}
\caption{Handle moves of the Akbulut cork with two circles.}
\label{fig:PQ-Kcalc}
\end{figure}
\begin{figure}[!htbp]
\centering
\includegraphics[scale=0.78]{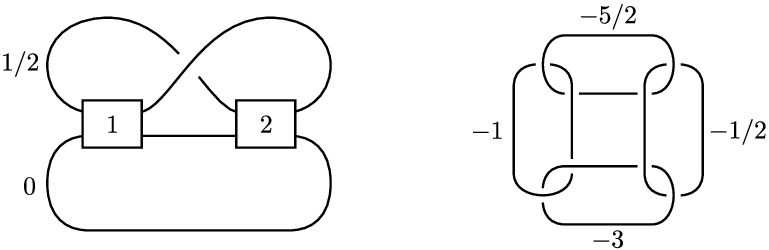}
\caption{Surgery diagrams of $\partial{Q}$.}
\label{fig:bdryQ}
\end{figure}

\begin{rem}
 The trisection genus of $P$ is greater than or equal to $4$.
 By Lemma~\ref{lem:rtd-Euler}, we see that $P$ could admit relative trisections of $(4,3;0,4)$ or $(4,2;1,1)$. 
 However, we have not been able to construct such relative trisections.
 If $P$ cannot admit genus $4$ relative trisections, the trisection genus for $4$-manifolds with boundary is not homeomorphism invariant.
\end{rem}

\subsection*{Acknowledgements}%
 This paper is partially based on the author's master thesis.
 The author would like to thank his adviser Kouichi Yasui for his encouragement and helpful discussions.
 The author was partially supported by JST SPRING, Grant Number JPMJSP2138.
%


\end{document}